\documentclass{amsart}

\usepackage{amsthm, amssymb, amsfonts, amsmath, tikz}

\newtheorem{theorem}{Theorem}[section]

\newtheorem{corollary}[theorem]{Corollary}
\newtheorem{proposition}[theorem]{Proposition}
\theoremstyle{definition}
\newtheorem{definition}[theorem]{Definition}
\theoremstyle{remark}

\newtheorem{problem}[]{Problem}
\newtheorem{conjecture}{Conjecture}
\newtheorem{observation}{Observation}
\newcommand{\sv}{\chi_{svtla}}
\newcommand{\se}{\chi_{setla}}

\title{Super Total Local Antimagic Coloring of Graphs}
\author{Ravindra Pawar}
\author{Tarkeshwar Singh}
\thanks{Ravindra Pawar: p20200020@goa.bits-pilani.ac.in}
\thanks{Tarkeshwar Singh: tksingh@goa.bits-pilani.ac.in}

\begin{document}

\maketitle

\begin{abstract}
Let $G = (V,E)$ be a finite simple undirected graph without isolated vertices. A bijective map $f: V \cup E \rightarrow \{1,2, \dots, |V|+ |E| \}$ is called total local antimagic labeling if for each edge $uv \in E, w(u) \ne w(v)$, where $w(v)$ is a weight of a vertex $v$ defined by $w(v) = \sum_{x \in NT(v)} f(x)$, where $NT(v) = N(v) \cup \{uv: uv\in E\}$ is the total open neighborhood of a vertex $v$. Further, $f$ is called \textit{super vertex total local antimagic labeling} or \textit{super edge total local antimagic labeling} if $f(V) = \{1,2, \dots, |V|\}$ or $f(E) = \{1,2, \dots, |E|\}$, respectively. The labeling $f$ induces a proper vertex coloring of $G$. The \textit{super vertex (edge) total local antimagic chromatic number} of a graph $G$ is the minimum number of colors used over all colorings of $G$ induced by super vertex (edge) total local antimagic labeling of $G$. In this paper, we discuss the super vertex (edge) total local antimagic chromatic number of some families of graphs.
\end{abstract}
{\setlength{\parindent}{0cm}\textbf{Keywords:} Antimagic Graph, Local Antimagic Graph, Local Antimagic Chromatic Number, Super Vertex (Edge) Total Labeling.}
\section{Introduction}
Throughout the paper, we assume that $G = (V,E)$ is a simple finite undirected graph without isolated vertices. By a \textit{labeling} of a graph, we mean a one-to-one correspondence between the labeling set and the elements of a graph (vertices or edges or both).
Accordingly, there are many variants of graph labelings viz vertex labeling, edge labeling, super vertex/edge labeling, etc.  We refer West \cite{west} for graph theoretic terminology and notations.\\

Hartsfield and Ringel \cite{pearls_in_graph_theory} introduced the concept of \textit{antimagic labeling} of a graph $G$: let $f:E \to \{1,2, \dots, |E|\}$ be a bijective map. The weight of a vertex $v$ is defined by $w(v):= \sum_{uv \in E}f(uv)$ and for given any pair of vertices $u,v$ their weights are distinct, i.e. $w(u) \ne w(v)$. They conjectured that ``any graph $G$ other than $K_2$ is antimagic," which is open till date. \\

Arumugam et al. \cite{lac_arumugam} and Bensmail et al. \cite{lac_bensmail} independently introduced a local version of antimagic labeling of a graph $G$. The edge labeling $f$ is said to be local antimagic labeling if $w(u) \ne w(v)$, whenever $uv$ is an edge in $G$.\\

Once we treat vertex weights as colors, the local antimagic labeling naturally induces a proper vertex colouring. The \textit{local antimagic chromatic number} of a graph $G$, denoted by $\chi_{la}(G)$ is the minimum number of colors used over all colorings of $G$ induced by local antimagic labeling of $G$.\\

A bijective map $f: V \cup E \to \{1,2,\dots, |V| + |E|\}$ is called total labeling.
Recently researchers extended the notion of local antimagic chromatic number by calculating weights of vertices induced by local antimagic total labelings of graphs, where $w(v) =f(v) + \sum_{uv \in E} f(uv)$ (see,  \cite{tlac_hadiputra, tlac_slamin, tlac_putri}).\\

Motivated by the above, we define the following:
\begin{definition}
The \textit{total open neighborhood} of a vertex $u \in V$ is denoted by $NT(u)$ and is given by $NT(u) = N(u) \cup \{uv: uv\in E\}$, where $N(u)$ is an open neighborhood of a vertex $u$ in a graph $G$ or in other words the total open neighborhood of a vertex $u$ in a graph is the collection of all vertices (except $u$) adjacent to $u$ together with all the edges incident to $u$.
\end{definition}
 The total closed neighborhood of a vertex $u$, denoted by $NT[u]$ is obtained by adding $u$ to $NT(u)$, i.e. $NT[u] = NT(u) \cup \{uv: uv \in E\}$.\\
 
\begin{definition}
Given a graph $G= (V,E)$. A bijective map $f: V \cup E \rightarrow \{1,2, \dots, |V| + |E| \}$ is called total local antimagic labeling if for each edge $uv \in E, w(u) \ne w(v)$, where $w(v)$ is a weight of $v$ given by $w(v) = \sum_{x \in NT(v)} f(x)$. Such a function $f$ is called a super vertex total local antimagic labeling if  $f(V) = \{1,2, \dots, |V|\}$ and it is called super edge total local antimagic labeling if $f(E) = \{1,2, \dots, |E|\}$.
\end{definition}

A super vertex (edge) total local antimagic labeling induces a proper vertex coloring of $G$ by considering the vertex weights as colours. Hence, we define the super vertex (edge) total local antimagic chromatic number as follows :\\

The \textit{super vertex total local antimagic chromatic number} of a graph $G$, denoted by $\sv(G)$, is the minimum number of colors used over all colorings of $G$ induced by super vertex total local antimagic labeling of $G$.\\

The \textit{super edge total local antimagic chromatic number} of a graph $G$, denoted by $\se(G)$, is the minimum number of colors used over all colorings of $G$ induced by super edge total local antimagic labeling of $G$. 
By definition we have $\chi(G) \le \sv(G)$ and $\chi(G) \le \se(G)$, respectively.\\

We abbreviate ``super vertex total local antimagic" as ``svtla" and ``super edge total local antimagic" as ``setla". In this paper, we discuss the super vertex (edge) total local antimagic chromatic number for a graph and study it for some graph families. 
This notion is a natural extension of the existing local antimagic chromatic number of graphs.\\

\begin{definition}
A magic rectangle $MR(m,n)$ of size $m \times n$ is a rectangular arrangement of first $mn$ natural numbers such that the sum of all entries in each row is the same and the sum of all entries in each column is the same.
\end{definition}
Harmuth gave the following theorem \cite{mr1}, which gives the necessary and sufficient conditions for the existence of a magic rectangle of a given order.
\begin{theorem}\cite{mr1} \label{mr}
A magic rectangle $MR(a,b)$ exists if and only if $a,b > 1$, $ab>4$, and $a \equiv b(\bmod\ 2)$.
\end{theorem}
If $G$ admit a svtla labeling $f$, then the sum $\sum_{x \in V } w(x)$ counts the label of each vertex $v$ exactly $deg(v)$ times. Also, the label of an edge $ e = uv$ is counted only in $w(u)$ and $w(v)$, i.e. exactly twice. Then we have the following observations.
\begin{observation} \label{ob:sv-counting}
If $G$ admits svtla labeling $f$ then $\sum_{x \in V} w(x) = \sum_{v \in V} deg(v)f(v) + 2\sum_{e \in E} f(e)$.
\end{observation}
We have a similar observation for setla graph.
\begin{observation} \label{ob:se-counting}
If $G$ admits a setla labeling $f$ then $\sum_{x \in V \cup E} w(x) = \sum_{v \in V}deg(v)f(v) + 2 \sum_{e \in E} f(e)$.
\end{observation}
In Section \ref{sec:svtla}, we study the super vertex total local antimagic labeling of graphs and in Section \ref{sec:setla}, we study super edge total local antimagic labeling of graphs.
\section{Super Vertex Total Local Antimagic Labeling} \label{sec:svtla}
\begin{proposition} \label{th:sv-leaves}
For any svtla graph $G$ with a vertex $v$ having the largest number of pendent vertices $l$, \(\sv(G) \ge l+1\).
\end{proposition}
\begin{proof}
Let $G$ be a graph on $n$ vertices and $v_1, v_2, \dots, v_l$ be the pendent vertices adjacent to \(v\). Let $f$ be any svtla labeling of $G$. Then the  weights of pendent vertices \(w(v_i) = f(v) + f(vv_i)\) are all distinct and $w(v) \ne w(v_i)$ for each $ i, \; 1 \le i \le l$. Hence, $f$ induces a proper $l+1$ vertex coloring of $G$. This proves the proposition.
\end{proof}

The following corollary is evident by Proposition \ref{th:sv-leaves}.
\begin{corollary} \label{th:sv-star}
For star $\sv(K_{1,n}) = n + 1$.
\end{corollary}
\begin{proof}
By Proposition~\ref{th:sv-leaves}, $\sv(K_{1,n}) \ge n+1$. We define super vertex total local antimgaic labeling of $K_{1,n}$  as $f(c) = n+1$ and $f(v_i) = i$ and $f(cv_i) = n+1+i $  for each $i, \; 1\leq i \leq n $ as shown in Figure~\ref{fig:sv-star}, therefore $\sv(K_{1,n}) \le n+1$. This proves that $\sv(K_{1,n}) = n+1$.  
\end{proof}
\begin{figure}[ht]
    \centering
    \includegraphics[scale = 0.5]{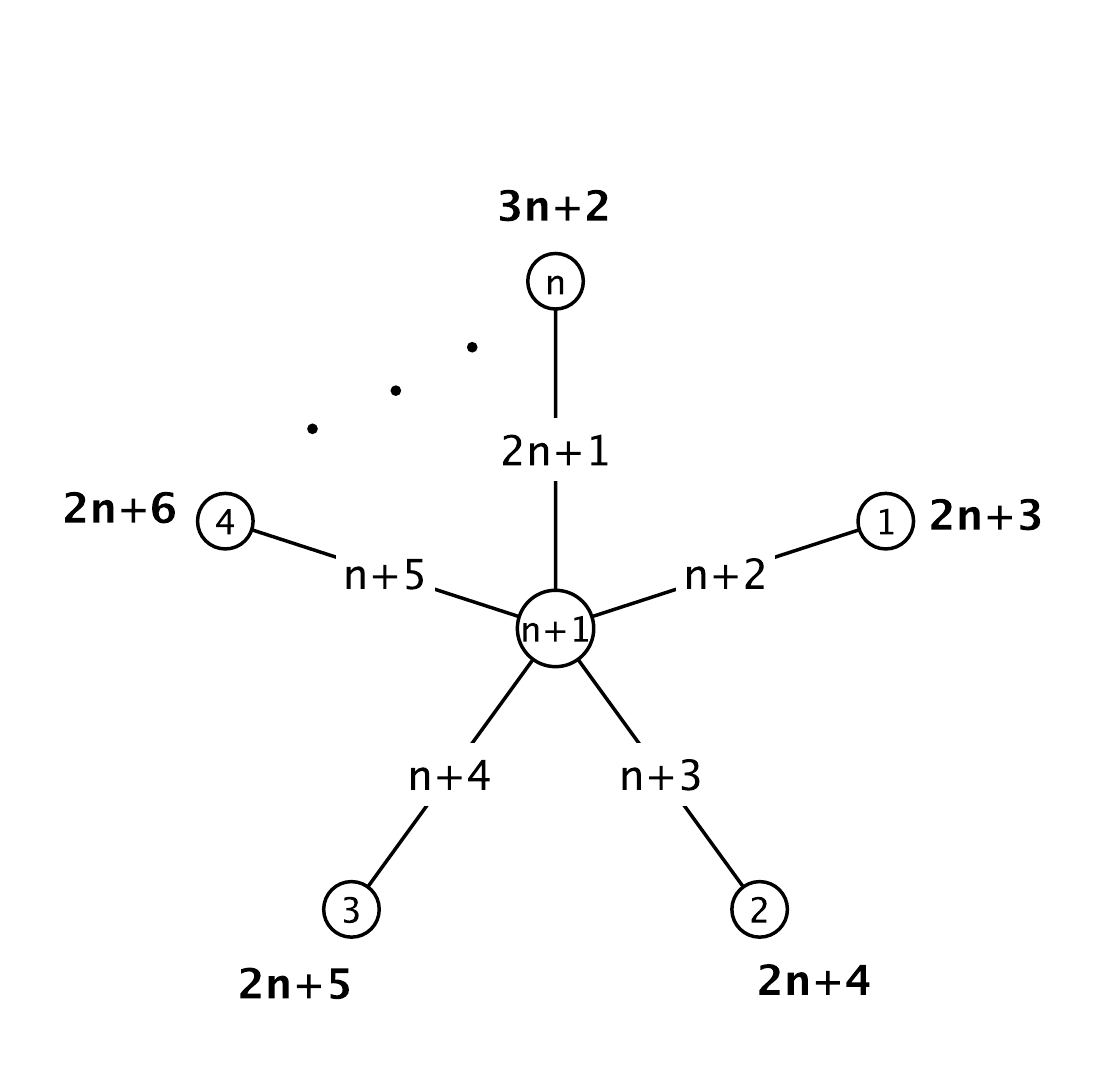}
    \caption{svtla labeling of a star.}
    \label{fig:sv-star}
\end{figure}
\begin{corollary}
If a tree $T$  with the largest number of pendent vertices equals $l$ at a vertex, $\sv(T) \ge l+1$.
\end{corollary}
\begin{theorem} \label{th:sv-path}
For a path $P_n$ with $n \ge 3$, $3 \le \sv(P_n) \le 5$.
\end{theorem}
\begin{proof}
It is easy to see that $\sv(P_2)=2$ and $\sv(P_3) = 3$ as shown in Figure \ref{fig:sv-path}. Let $n \ge 4$ and $P_n$ be path on $n$ vertices with the vertex set $\{v_1,v_2, \dots, v_n\}$ and $e_i = v_iv_{i+1}$ be edges. First, we will prove the lower bound. Let $f$ be svtla labeling of $P_n$. Then $w(v_1) = f(e_1) + f(v_2)$ and $w(v_3) = f(v_2) + f(v_4) + f(e_2) + f(e_3)$. If $w(v_1) = w(v_3)$ then we get $f(e_1) = f(e_2) + f(e_3) + f(v_4)$. Now the value of $f(e_1)$ is at most $2n-1$ and if we assign minimum labels to $e_2, e_3, v_4$ we get $f(e_2) + f(e_3) + f(v_4) = 2n + 4 > 2n-1 = f(e_1)$. This contradiction proves that $w(v_1) \ne w(v_3)$. Also, $w(v_1) \ne w(v_2)$ and $w(v_2) \ne w(v_3)$ since $v_1v_2, v_2v_3 \in E(P_n)$ and $f$ is svtla labeling. Hence, for $n \ge 4$,
\begin{equation} \label{eq: svtlaPn}
    \chi_{svtla}(P_n) \ge 3.
\end{equation}
\\
Now for the upper bound, we have the following two cases:\\  
\\
{\bf Case-1:} When $n \equiv 3(\bmod~4)$, then we define super vertex total labeling $f$ by\\
\begin{align*}
        f(v_i) &= \begin{cases}
                i-1   &\text{ if } i \equiv 0(\bmod~4)\\
                i   &\text{ if } i \equiv 1 \mbox{ or } 2(\bmod~4)\\
                i+1   &\text{if } i \equiv 3(\bmod~4) \text{ and } i \ne n\\
                n  &\text{if } i=n
                \end{cases} 
\end{align*}
and $f(e_i) = 2n-i$. Therefore, $w(v_1) = f(v_2) + f(e_1) = 2+(2n-1) = 2n+1$, $w(v_{n-1}) = f(v_{n-2}) + f(v_{n}) + f(e_{n-2}) + f(e_{n-1}) = (n-2)+n+(2n-n+2)+(2n-n+1) = 4n+1$, and $w(v_n) = f(v_{n-1}) + f(e_{n-1}) = (n-1)+(2n-n+1) = 2n$. Now for each $i,~ 2 \le i \le n-1$,
\begin{align*}
    w(v_i) &= f(v_{i-1}) + f(v_{i+1}) + f(e_{i-1}) + f(e_{i})\\
           &= \begin{cases}
           i + (i+1) + (2n-(i-1)) + (2n-i) \quad & \text{if $i \equiv 0(\bmod~4)$}\\
           (i-2)+(i+1) + (2n-(i-1)) + (2n-i) \quad & \text{if $i \equiv 1(\bmod~4)$}\\
           (i-1)+(i+2) + (2n-(i-1)) + (2n-i) \quad & \text{if $i \equiv 2(\bmod~4)$}\\
           (i-1)+i + (2n-(i-1)) + (2n-i) \quad & \text{if $i \equiv 3(\bmod~4)$}
           \end{cases}\\
           &= \begin{cases}
           4n & \text{if $i$ is odd}\\
           4n+2 & \text{if $i$ is even}.
           \end{cases}
\end{align*} 
Therefore, weights of adjacent vertices are distinct, i.e. $f$ is svtla labeling of $P_n$, and it induces proper $5$-coloring of $P_n$.\\
\\
{\bf Case-2:} When $n \not \equiv 3(\bmod~4)$, then we  define super vertex total labeling $f$ by\\
\begin{align*}
        f(v_i) &= \begin{cases}
                i-1  \quad & \text{ if } i \equiv 0(\bmod~4) \text{ and }\\
                i  \quad & \text{ if } i \equiv 1 \mbox{ or }2(\bmod~4) \text{ and }\\
                i+1  \quad & \text{if } i \equiv 3(\bmod~4).
                \end{cases} 
\end{align*}
and $f(e_i) = 2n-i$. Therefore, $w(v_1) = f(v_2) + f(e_1) = 2+(2n-1) = 2n+1$,
\begin{align*}
    w(v_n)  &= f(v_{n-1}) + f(e_{n-1})\\
            &= \begin{cases}
            n + (2n - (n - 1)) & \text{if $n \equiv 0(\bmod~4)$}\\
            (n-2) + (2n - (n - 1)) & \text{if $n \equiv 1(\bmod~4)$}\\
            (n-1) + (2n - (n - 1)) & \text{if $n \equiv 2 \text{ or }3(\bmod~4)$}
            \end{cases}\\
            &= \begin{cases}
            2n+1 & \text{if $n \equiv 0(\bmod~4)$}\\
            2n-1 & \text{if $n \equiv 1(\bmod~4)$}\\
            2n & \text{if $n \equiv 2 \text{ or } 3(\bmod~4)$},
            \end{cases}
\end{align*}
and for each $i,~ 2 \le i \le n-1$,
\begin{align*}
    w(v_i) &= f(v_{i-1}) + f(v_{i+1}) + f(e_{i-1}) + f(e_{i})\\
           &= f(v_{i-1}) + f(v_{i+1}) + (2n-(i-1)) + (2n-i))\\
           &= (4n-2i+1) + f(v_{i-1}) + f(v_{i+1})\\
           &= \begin{cases}
           (4n-2i+1) + i + (i+1) & \text{if $i \equiv 0(\bmod~4)$}\\
           (4n-2i+1) + (i-2)+(i+1) & \text{if $i \equiv 1(\bmod~4)$}\\
           (4n-2i+1) + (i-1)+(i+2) & \text{if $i \equiv 2(\bmod~4)$}\\
           (4n-2i+1) + (i-1)+i & \text{if $i \equiv 3(\bmod~4)$}
           \end{cases}\\
           &= \begin{cases}
           4n & \text{if $i$ is odd}\\
           4n+2 & \text{if $i$ is even}.
           \end{cases}
\end{align*}
Consider the following two sub-cases:\\
\\
\textbf{Subcase-(i):} When $n \equiv 0(\bmod~4), w(v_1) = w(v_n) = 2n+1$, and for each  $i,\; 2 \le i \le n-1, w(v_i) = 4n$ or $4n+2$. Hence, $\sv(P_n) \le 3$. Also, by Equation \ref{eq: svtlaPn}, $\sv(P_n) \ge 3$. Therefore, $\sv(P_n) = 3$.\\
\\
\textbf{Subcase-(ii):} When $n \equiv 1 \text{ or }2(\bmod~4)$, $w(v_1) = 2n+1, w(v_n) = 2n-1$ and for each $i, \; 2\le i \le n-1, w(v_i) = 4n \text{ or } 4n+2$. Therefore $\sv(P_n) \le 4$.\\
\\
Hence,  $\sv(P_n) = 3$, when $n \equiv 0(\bmod~4)$ and
\begin{equation*}
\sv(P_n) \le
\begin{cases}
4 & \text{ if } n \equiv 1 \text{ or }2(\bmod~4)\\
5 & \text{ if } n \equiv 3(\bmod~{4}).
\end{cases}
\end{equation*}
This completes the proof.
\end{proof}
\begin{figure}
    \centering
    \includegraphics[scale = 0.6]{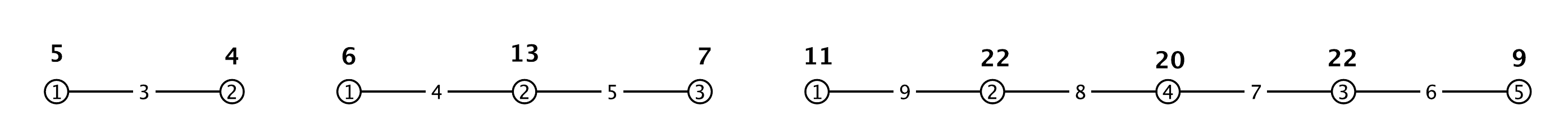}
    \caption{svtla of $P_2, P_3$ and $P_5$.}
    \label{fig:sv-path}
\end{figure}
\begin{theorem} \label{th:sv-cycle}
For  a cycle $C_n$ with $ n \ge 3,\; \sv(C_n) \le 4$.
\end{theorem}
\begin{proof}
Let $C_n$ be cycle on $n \ge 3$ with vertex set $\{v_1,v_2, \dots, v_n\}$ and $e_i = v_iv_{i+1}$ be edges, where subscripts are taken modulo $n$.\\
\\
Now we  have the following two cases:\\
\\
\textbf{Case 1:} When $n \equiv 3(\bmod~4)$, then we define super vertex total labeling $f$ by
\begin{align*}
    f(v_i) &= 
\begin{cases}
                i-1  \quad & \text{if } i \equiv 0(\bmod~4) \\
                i  \quad & \text{if } i \equiv 1 \mbox{ or }2(\bmod~4)\\
                i+1  \quad & \text{if } i \equiv 3(\bmod~4), i \ne n\\
                n , \quad & \text{if } i=n
                \end{cases}
\end{align*}
and
\begin{align*}
    f(e_i) =\begin{cases}
            2n-i+1 \quad & \text {if } i \ne n\\
            n+1 \quad & \text{if } i = n.
            \end{cases}
\end{align*}
Therefore,
\begin{align*}
    w(v_1) &=f(v_2) + f(v_n) + f(e_1) + f(e_n)\\
           &=2 + n + 2n + (n+1)\\
           &=4n+3\\
    w(v_n) &= f(v_1) + f(v_{n-1}) + f(e_{n-1}) + f(e_n)\\
           &= 1 + (n-1) + (n+2) + (n+1)\\
           &= 3n+3.
\end{align*}
Now for each $i, ~2 \le i \le n-1$,
\begin{align*}
    w(v_i)  &= f(v_{i-1}) + f(v_{i+1}) + f(e_{i-1}) + f(e_{i})\\
            &= (4n-2i+3) + f(v_{i-1}) + f(v_{i+1})\\
            &= \begin{cases}
            (4n-2i+3) + i + (i+1) & \text{if } i \equiv 0(\bmod~4)\\
            (4n-2i+3) + (i-2) + (i+1) & \text{if } i \equiv 1(\bmod~4)\\
            (4n-2i+3) + (i-1) + (i+2) & \text{if } i \equiv 2(\bmod~4)\\
            (4n-2i+3) + (i-1) + i &  \text{if } i \equiv 3(\bmod~4)
            \end{cases}\\
            &= \begin{cases}
            4n+4 & \text{if $i$ is even}\\
            4n+2 & \text{if $i$ is odd}.
            \end{cases}
\end{align*}
It is easy to see that  $f$ is svtla labeling of $C_n$ and it induces $4$ proper vertex coloring of $C_n$.\\
\\
\textbf{Case 2:} When $n \not \equiv 3(\bmod~4)$, then we define super vertex total labeling $f$ by
\begin{align*}
    f(v_i) =\begin{cases}
            i-1  \quad &\text{ if } i \equiv 0(\bmod~4) \\
            i  \quad &\text{ if } i \equiv 1 \text{ or } 2(\bmod~4)\\
            i+1  \quad &\text{if } i \equiv 3(\bmod~4)
            \end{cases}
\end{align*}
and
\begin{align*}
    f(e_i)=\begin{cases}
            2n-i \quad &\text { if } i \ne n\\
            2n \quad &\text{ if } i = n.
            \end{cases} 
\end{align*}
Therefore,
\begin{align*}
w(v_1)  &=f(v_2) + f(v_n) + f(e_1) + f(e_n)\\
        &=2 + f(v_n) + (2n-1) + 2n\\
        &=4n + 1+ f(v_n) \\
        &=\begin{cases}
        4n + 1 +  (n-1) \quad &\text{if } n \equiv 0(\bmod~4)\\
        4n + 1+ n \quad &\text{if } n \equiv 1 \text{ or } 2(\bmod~4)\\
        \end{cases}\\
        &=\begin{cases}
        5n \quad &\text{if } n \equiv 0(\bmod~4)\\
        5n +1 \quad &\text{if } n \equiv 1 \text{ or } 2(\bmod~4)
        \end{cases}
\end{align*}
and 
\begin{align*}
  w(v_n)&= f(v_1) + f(v_{n-1}) + f(e_{n-1}) + f(e_n)\\
        &= 1 + f(v_{n-1}) + (n+1) + 2n\\
        &= (3n + 2) + f(v_{n-1})\\
        &= \begin{cases}
        (3n + 2) + n \quad &\text{if } n \equiv 0(\bmod~4)\\
        (3n + 2) + (n-2) \quad &\text{if } n \equiv 1(\bmod~4)\\
        (3n + 2) + (n-1) \quad &\text{if } n \equiv 2(\bmod~4)
        \end{cases}\\
        &= \begin{cases}
        4n + 2 \quad &\text{if } n \equiv 0(\bmod~4)\\
        4n \quad &\text{if } n \equiv 1(\bmod~4)\\
        4n + 1 \quad &\text{if } n \equiv 2(\bmod4~).
        \end{cases}\\
\end{align*}
Now for each $i, \;2 \le i \le n-1$,
\begin{align*}
    w(v_i) &= f(v_{i-1}) + f(v_{i+1}) + f(e_{i-1}) + f(e_{i})\\
           &= f(v_{i-1}) + f(v_{i+1}) + (2n-i+1) + (2n-i)\\
           &= (4n-2i+1) + f(v_{i-1}) + f(v_{i+1})\\
           &= \begin{cases}
            (4n-2i+1) + i + (i+1) & \text{if } i \equiv 0(\bmod~4)\\
            (4n-2i+1) + (i-2) + (i+1) & \text{if } i \equiv 1(\bmod~4)\\
            (4n-2i+1) + (i-1) + (i+2) & \text{if } i \equiv 2(\bmod~4)\\
            (4n-2i+1) + (i-1) + i &  \text{if } i \equiv 3(\bmod~4)
            \end{cases}\\
           &= \begin{cases}
            4n+2 & \text{if $i$ is even}\\
            4n & \text{if $i$ is odd}.
            \end{cases}
\end{align*} 
\\
This proves that $f$ is svtla labeling. Now, we have the following subcases:\\
\\
\textbf{Subcase (i):} When $n \equiv 0(\bmod~4)$, $w(v_1) = 5n, w(v_n) = 4n+2$ and for each $i,\;2 \le i \le n-1, w(v_i) = 4n \text{ or } 4n+2$. Hence, $\sv(C_n) \le 3$.\\
\\
\textbf{Subcase (ii):} When $n \equiv 1(\bmod~4)$, $w(v_1) = 5n + 1, w(v_n) = 4n$ and for each $i,\;2 \le i \le n-1, w(v_i) = 4n \text{ or } 4n+2$. Hence, $\sv(C_n) \le 3$. Also, $3 = \chi(C_n) \le \sv(C_n) \le 3$. This proves $\sv(C_n) = 3$.\\
\\
\textbf{Subcase (iii):} When $n \equiv 2(\bmod~4)$, $w(v_1) = 5n+1, w(v_n) = 4n+1$ and for each $i,\;2 \le i \le n-1, w(v_i) = 4n \text{ or } 4n+2$. Hence, $\sv(C_n) \le 4$.\\
This completes the proof. 
\end{proof}
\begin{figure}[ht]
    \centering
    \includegraphics[scale = 0.5]{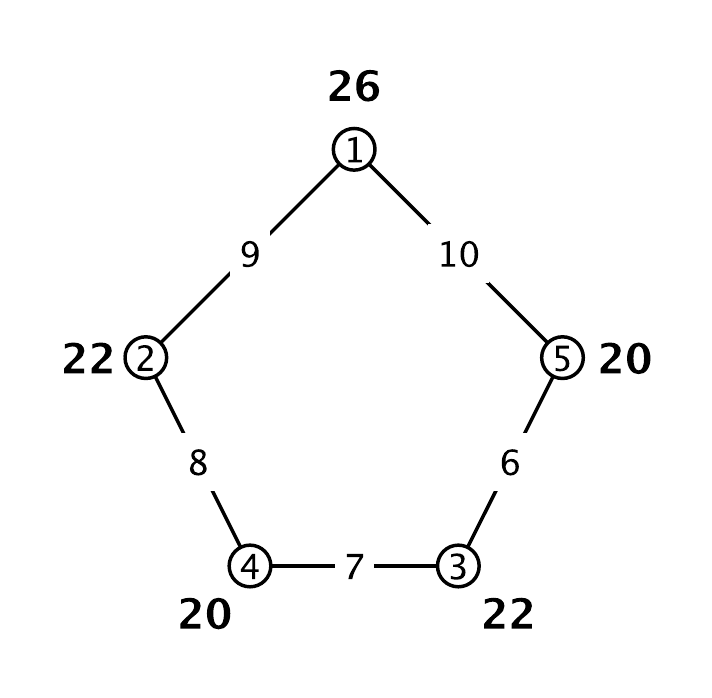}
    \caption{svtla labeling of $C_5$.}
    \label{fig:sv-c5}
\end{figure}
As observed in the Theorem \ref{th:sv-path}, for $n \equiv 0(\bmod~4)$, $\sv(P_n) = \chi(P_n) + 1$ and also in the Theorem \ref{th:sv-cycle}, for $n \equiv 1(\bmod~4)$, $\sv(C_n) = \chi(C_n)$. Then the following question arises.\\
\begin{problem}
Characterise graphs $G$ for which $\sv(G) = \chi(G)$.
\end{problem}
\begin{theorem}
For $n \ge 2$,  $ \sv(K_n) = n$.
\end{theorem}
\begin{proof}
We will prove the result by induction on $n$. By Theorem \ref{th:sv-path}, we know that $\sv(K_2) = 2$ and by Theorem \ref{th:sv-cycle}, $\sv(K_2 + K_1) = \sv(K_3) \equiv \sv(C_3) = 3$. Hence, the result is true for $n = 2$.
Assume that for a given $n \ge 3, \sv(K_n) = n$ with svtla labeling $f$.
Without loss of generality we may assume that $v_1, v_2, \dots, v_n$ are the vertices of \(K_n\) such that $w_f(v_1) < w_f(v_2)< \dots < w_f(v_n)$. Let $e_1, e_2, \dots, e_m$ be the edges of $K_n$, where $m = \frac{n(n-1)}{2}$. Let \(V(K_1) = v_0\). We define super vertex total labeling $g$ of \(K_n+K_1\) by 
\[
g(x) = 
\begin{cases}
1 & \text{ if } x = v_0\\
f(v_i)+1 & \text{ if } x = v_i, 1 \le i \le n\\
f(e_i)+1 & \text{ if } x = e_i, 1 \le i \le n\\
m+n+1+i & \text{ if } x = v_0v_i, 1 \le i \le n.
\end{cases}
\]
Therefore,
\begin{align*}
w_g(v_0)
&= \sum_{v \in V(G)} f(v) + \sum_{i=1}^{n} f(v_0v_i)\\
&= (2+3+ \dots + (n+1)) + \sum_{i=1}^{n} (m+n+1+i)\\
&= \frac{n^3+4n^2+7n}{2}\\
\end{align*}
and for \(1 \le i \le n\),
\begin{align*}
w_g(v_i)
&= \sum_{y \in NT_{K_n}(v_i)}f(y) + m+n+1+i\\
&=w_f(v_i) + 2(n-1) + (m+n+1+i).
\end{align*}
From the expression of $w_g(v_i)$, it is clear that $w_g(v_1) < w_g(v_2) < \dots < w_g(v_n)$. Now we show that $w_g(v_0) > w(v_i)$ for each $i, \;  1 \le i 
\le n$. Consider any vertex $v_i$ . If we assign largest labels from the set $\{1, 2, \dots, \frac{n(n+1)}{2}\}$ to the elements in $NT(v_i)$ then we obtain
\[
w_{g}(v_i) \le \left(\frac{n^2+n}{2} - (2n-3) \right) + \left(\frac{n^2+n}{2} - (2n-2)\right) + \dots + \frac{n^2+n}{2} = n^3 - 2n^2 + 4n - 3.
\]
Now for $n \ge 3$, $w_g(v_0) = \frac{n^3+4n^2+7n}{2} > n^3 - 2n^2 + 4n - 3 \ge w_g(v_i)$ where $1 \le i \le n$. Hence, all the vertex weights in $K_n+K_1$ are distinct. This proves the theorem.
\end{proof}
 Now we calculate the super vertex total local antimagic chromatic number of a complete bipartite graph $K_{m,n}$.
\begin{theorem} \label{th:sv-bipartite}
$\sv(K_{m,n}) = 2$ if and only if $m,n>1, mn >4$ and $m \equiv n (\bmod~2)$.
\end{theorem}
\begin{proof}
Suppose $m,n>1, mn >4$ and $m \equiv n (\bmod~2)$. Then by Theorem~\ref{mr}, there exists a magic rectangle $MR(m,n)$ with row sum $\frac{n(mn+1)}{2}$, column sum $\frac{m(mn+1)}{2}$. Let $\{x_1, x_2, \dots, x_m \}$ and $\{y_1, y_2, \dots, y_n \}$ be bipartition of the vertex set of $K_{m,n}$ with $m \le n$. Adding $(m+n)$ to the each entry of $MR(m,n)$ we obtain a new rectangle say $R = [r_{i,j}]_{m\times n}$ in which row sum is $\rho = n(m+n)+\frac{n(mn+1)}{2}$ and column sum is $\sigma = m(m+n) + \frac{m(mn+1)}{2}$. Define the super vertex total labeling $f$ by
\begin{align*}
f(a) = \begin{cases}
        i &\mbox{ if } a = x_i\\
        m+i &\mbox{ if } a = y_i\\
        r_{i,j} &\mbox{ if } a = x_iy_j
        \end{cases}
\end{align*}
Then for any $i,~1 \le i \le m$ and for any $j,~1 \le j \le n$,
\begin{align*}
w(x_i)
&= \sum_{j=1}^{n} f(y_j) + \rho \\
&= \left(mn + \frac{n(n+1)}{2}\right) + \left(mn + n^2 + \frac{mn^2+n}{2}\right)\\
&= \frac{(2+m)n^2 + (2+4m)n}{2}\\
w(y_j)
&= \sum_{i=1}^{m} f(x_i) + \sigma \\
&= \frac{m(m+1)}{2} + \left(m^2 + mn + \frac{m^2n+m}{2} \right)\\
&= \frac{(n+3)m^2 + (2n+2)m}{2}
\end{align*}
It is easy to observe that weights of vertices in independent sets are the same and $w(x_i) < w(y_j)$ . Therefore $f$ is svtla labeling, and it induces $2$ colors. Hence, $\sv(K_{m,n}) \le 2$. We know, $2 = \chi(K_{m,n}) \le \sv(K_{m,n})$. Therefore, $\sv(K_{m,n}) = 2$.\\

Conversely, suppose that $\sv(K_{m,n}) = 2$ with svtla labeling $f$. We must have $w(x_i) \ne w(y_j)$ for any $i,\;1 \le i \le m$ and for any $j,\;1 \le j \le n$ and weights of vertices in the independent sets must be same. We can form a magic rectangle $MR(m,n)$ with $(i,j)$th entry as $f(x_iy_j) - (m+n)$ so that row sum is equal to $w(y_j)-n(m+n)$ (where $w(y_j)$ is same for each $j, 1\le j \le n$) and the column sum is equal to $w(x_i)-m(m+n)$ (where $w(x_i)$ is same for each $i, 1\leq i \leq m$). Hence, by Theorem  \ref{mr}, $m$ and $n$ satisfy the required conditions. This completes the proof.
\end{proof}
\begin{proposition}
For any $n \ge 1$,
\begin{align*}
\sv(K_{2,n}) = 
\begin{cases}
2 \mbox{ if }  n \ge 4 \mbox{ and n is even}\\
3 \mbox{ if } n \mbox{ is odd or } n=2.
\end{cases}
\end{align*}
\end{proposition}
\begin{proof}
We know that $\sv(K_{2,1}) = \sv(K_{2,2}) = 3$. Let $n \ge 3$. If $n$ is even by Theorem~\ref{th:sv-bipartite}, $\sv(K_{2,n}) = 2$. Let $n$ be odd. Again by Theorem~\ref{th:sv-bipartite}, $\sv(K_{2,n}) \ge 3$. We will show that $\sv(K_{2,n}) \le 3$. Let $\{x,y\}$ and $\{u_1, u_2, \dots, u_n\}$ be bipartition of $K_{2,n}$. By Corollary~\ref{th:sv-star}, $\sv(K_{2,1}) = \sv(P_3) = 3$ and by Theorem~\ref{th:sv-cycle}, $\sv(K_{2,2}) = \sv(C_4) \le 3$. For $n \ge 3$, define  super vertex total labeling of $K_{2,n}$ by
$f(x) = 1, f(y) = 2$, for each $i,~1 \le i \le n$, $f(u_i) = 2+i$, and 
\begin{align*}
    f(au_i) = \begin{cases}
            n+2+i & \text{ if } a = xu_i\\
            3n+3-i & \text{ if } a = yu_i
            \end{cases}
\end{align*}
Now we calculate the vertex weights:
For each $i, \; 1 \le i \le n$, $w(u_i) = 4n+8$, $w(x) = 2n^2 + 5n, w(y) = 3n^2 + 5n$. Observe that, $w(x) < w(u_i) < w(y)$. Hence, $f$ is svtla labeling and induces $3$ colors. Therefore, $\sv(K_{2,n}) \le 3$. Hence, $\sv(K_{2,n}) = 3$.  
\end{proof}
Let $S_{n,t}$ be a graph obtained by replacing each edge of a star $S_n$ by a path of length $t+1$ (see Figure \ref{fig:sv-sub_star}). We calculate the svtla chromatic number of $S_{n,t}$ for $t=1,2$. Let $vv_i$ be replaced by path $vv_{i,1},v_{i,2}, \dots ,v_{i,t+2}$.
\begin{theorem}
$\sv(S_{n,t}) \le n + t+1$, for $t = 1,2$.
\end{theorem}
\begin{proof}
We define the super vertex total labeling of $S_{n,t}$ by
\begin{align*}
&f(v) = nt+n+1,\\
&f(v_{i,j}) = i + n(j-1) \quad &\mbox{ if } 1 \le i \le n, 1 \le j \le t+1,\\
&f(vv_{i,1}) = n(2t+1)+1+i \quad &\mbox{ if } 1 \le i \le n,\\
&f(v_{i,j-1}v_{i,j}) = n(2t+3-j)-i+2 \quad &\mbox{ if } 1 \le i\le n, 2 \le j \le t+1.
\end{align*}
The weights of vertices induced by $f$ are:
\begin{align*}
    w(x) &= \sum_{i=1}^{n}f(v_{i,1}) + \sum_{i=1}^{n}f(v_{i,1}v{i,2})\\
    &= \sum_{i=1}^{n} i + \sum_{i=1}^{n} (2nt+n+1+i)\\
    &= 2n^2(1 + t) + 2n
\end{align*}
and for each $i,~ 1 \le i \le n $,
\begin{align*}
    w(v_{i,1}) &= f(v) + f(v_{i2}) + f(vv_{i,1}) + f(v_{i,2}v_{i,1})\\
    &= (nt+n+1) + (n+i) + (2nt+n+1+i) + (2nt+n+2-i)\\ 
    &= n(5t+4)+4+i,\\
    w(v_{i,2}) &= f(v_{i,1}) + f(v_{i,3}) + f(v_{i,2}v_{i,1}) + f(v_{i,2}v_{i,3})\\
    &= (i) + (i+2n) + (2nt+n+2-i) + (2nt+2-i)\\
    &= 4nt+3n+4,\\
    w(v_{i,3}) &= f(v_{i,2}) + f(v_{i,3}v_{i,2})\\
    &= (n+i) + (2nt+2-i)\\
    &= 2nt+n+2.
\end{align*}
Clearly, $f$ is a svtla labeling of $S_{n,t}$, and $f$ induces $n+t+1$ distinct colors. Hence, $\sv(S_{n,t}) \le n+t+1$.
\end{proof}
\begin{figure}[ht]
    \centering
    \includegraphics[scale = 0.5]{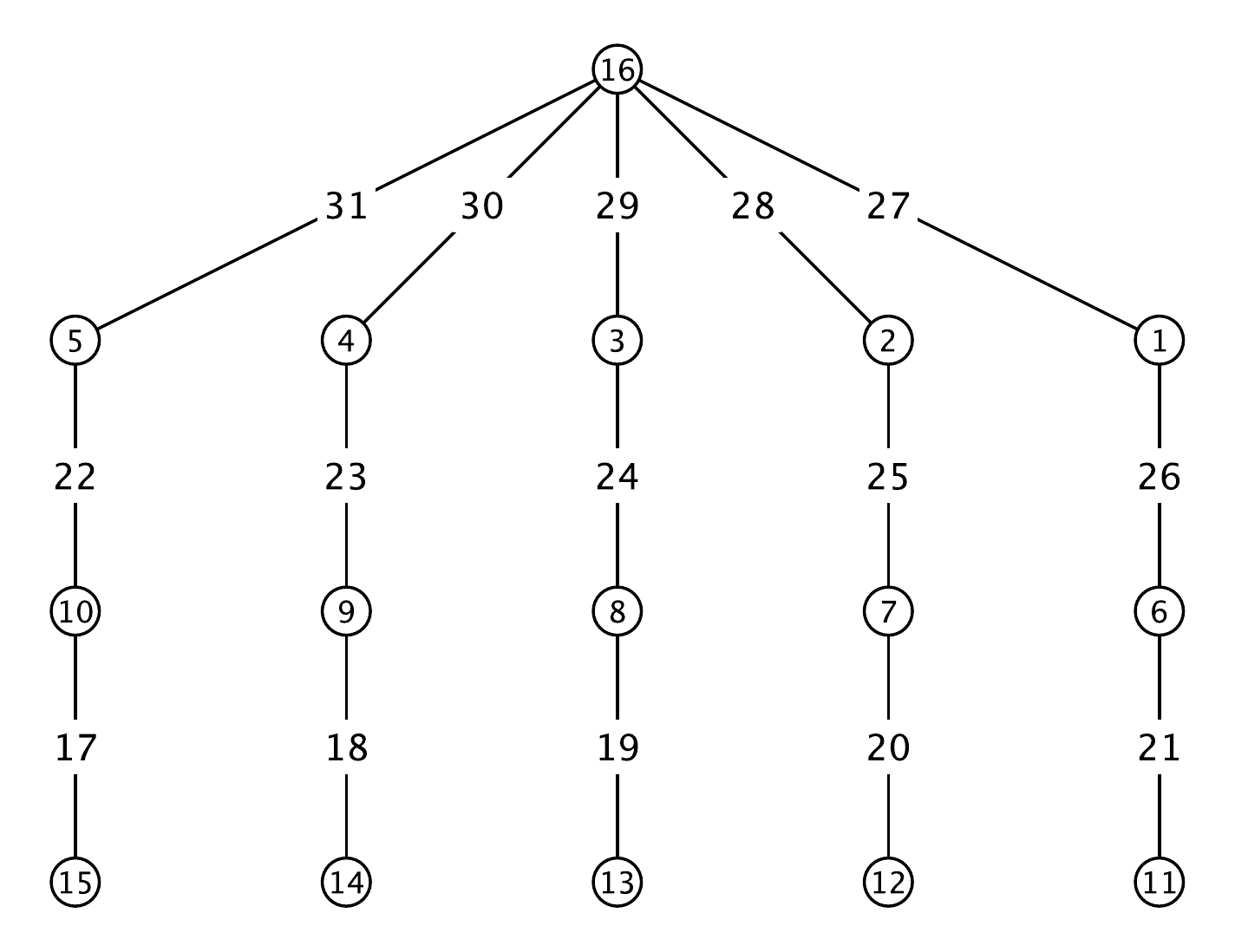}
    \caption{svtla labeling of $S_{5,2}$.}
    \label{fig:sv-sub_star}
\end{figure}
Let $B_{m,n}$ be the bi-star with vertex set $\{x,y, x_i, y_j,: 1\le i \le m, 1\le j \le n \}$ with centers $x$ and $y$ and edge set $\{xy, xx_i, yy_j : 1 \le i \le m, 1 \le j \le n\}$ . We calculate the super vertex total local anitmagic chromatic number $\sv(B_{m,n})$.
\begin{theorem}
For the bi-star $B_{m,n}, \sv(B_{m,n}) = n+2$.
\end{theorem}
\begin{proof}
	Let $m \le n$. By the Theorem \ref{th:sv-leaves}, $n+1 \le \se(B_{m,n})$. First we show that $\sv(B_{m,n}) > n+1$. Let $f$ be any svtla labeling of $B_{m,n}$. We know that weights of all pendent vertices $y_i$ are distinct and $w(y) \ne w(y_i)$ for any $i, ~1\le i \le n$. If $w(x) = w(y_i)$ then we obtain $f(x) + f(xx_i) = f(x) + f(xy) + \sum_{i=1}^{n}(f(y_i) + f(yy_i))$. This implies, $f(xx_i) = f(xy) + \sum_{i=1}^{n}(f(y_i) + f(yy_i))$ which is not possible. Therefore, $\sv(B_{m,n}) \ge n+2$. For the upper bound, we define super vertex total labeling $f$ of $B_{m,n}$ as:
	\begin{align*}
		f(x) &= m+n+2\\
		f(y) &= n+2\\
		f(xy) &= 2m+2n+3\\
		f(x_i) &= i&1 \le i \le m\\
		f(xx_i) &= m+n+2+i &1 \le i \le m\\
		f(yy_i) &=  2m+n+2+i&1 \le i \le n
	\end{align*}
	and we label the vertices $y_i$ by $\{m+1, m+2, \dots, m+n, m+n+1\} - \{n+2\}$ in any manner. The sum of these $y_i$ labels is $\frac{n^2+2mn+2m+n-2}{2}$. Now we calculate the weights:
	\begin{align*}
		w(x) &= f(y) + f(xy) +  \sum_{i=1}^{m} (f(x_i) + f(xx_i))\\
				&= (n+2) + (2m+2n+3) + \sum_{i=1}^{m} (m+n+2+2i)\\
				&= 2m^2+mn+5m+3n+5,\\
		w(y) &= f(x) + f(xy) +  \sum_{i=1}^{n} f(y_i) + \sum_{i=1}^{n}f(yy_i)\\
				&= (m+n+2) + (2m+2n+3) + \frac{n^2+2mn+2m+n-2}{2} + \sum_{i=1}^{n}(2m+n+2+i)\\
				&= 2n^2+3mn+4m+6n+4.
	\end{align*}
	The weights of pendent vertices $w(x_i) = f(x) + f(xx_i) = m+n+4+i$ for each $i, ~1 \le i \le m$ and $w(y_j) = f(y) + f(yy_j) = 2m+2n+4+j$ for each $j, \; 1 \le j \le n$. Thus $f$ is a required svtla labeling which induces $n+2$ colors. Hence, $\sv(B_{m,n}) = n+2$.
\end{proof}
\begin{figure}[ht]
    \centering
    \includegraphics[scale = 0.5]{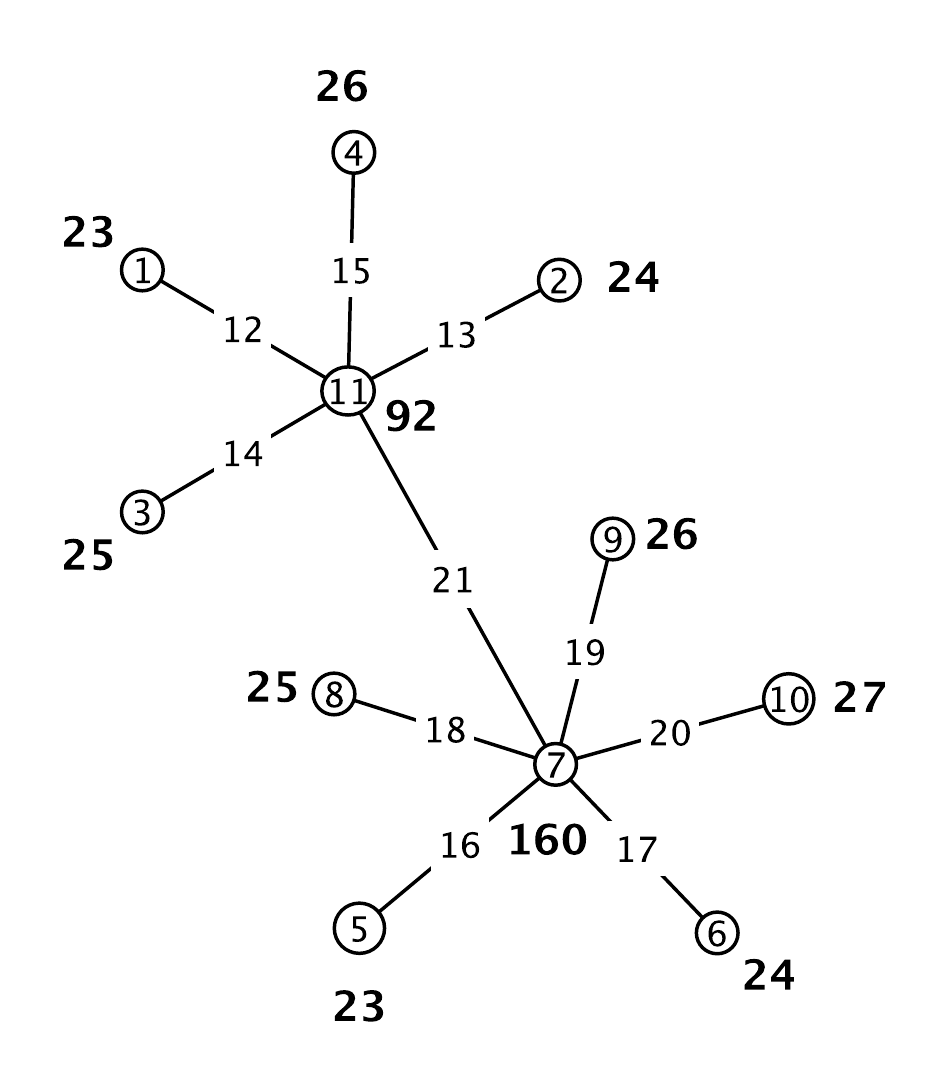}
    \caption{svtla labeling of bi-star.}
\end{figure}
\begin{theorem} \label{th:sv-corona}
If $G$ is an $r$-regular graph then $\sv(G \circ K_1) \le \sv(G)+1$.
\end{theorem}
\begin{proof}
Let $G$ be an $r$-regular graph with $V(G) = \{v_1,v_2, \dots, v_n\}$ and $m$ edges. Let $f$ be svtla labeling of $G$. Without loss of generality, we may assume that $f(v_i) = i$, for $1 \le i \le n$. Let $\{x_1, x_2, \dots, x_n\}$ be newly added vertices to obtain $H = G \circ K_1$ such that there is an edge $e_{i} = v_ix_i$ for each $i,\;1 \le i \le n$. Define super vertex total labeling $g$ of $G \circ K_1$ by
\begin{align*}
&g(v_i) = f(v_i) = i \quad &1 \le i \le n\\
&g(x_i) = n + i \quad &1 \le i \le n\\
&g(e_i) = 2n+m-i+1 \quad &1 \le i \le m\\
&g(e) = 2n + f(e) \quad &\mbox{where } e \in E(G).
\end{align*}
 Now we calculate the weight of each vertex in $G \circ K_1$ induced by super vertex total labeling  $g$.
For each $i, ~1 \le i \le n$, $w_g(x_i) = g(v_i) + g(e_{i}) = 2n+m+1$ and $g(x_i) + g(e_i) = 3n+m+1$. Now for any vertex $v_i \in V(G)$, 
\begin{align*}
w_g(v_i) &= \sum_{x \in NT_H(v_i)}g(x)\\
         &= g(x_i) + g(e_i) + \sum_{x \in NT_G(v_i)}g(x)\\
         &= (3n+m+1) + \sum_{x \in NT_G(v_i)}g(x)\\
         &= (3n+m+1) + \sum_{u \in N_G(v_i)}g(u) + \sum_{uv_i \in E(G)}g(uv_i)\\
         &= (3n+m+1) + \sum_{u \in N_G(v_i)}f(u) + \sum_{uv_i \in E(G)}(2n +  f(uv_i))\\
         &= (3n+m+1) + 2nr + \sum_{u \in N_G(v_i)}f(u) + \sum_{uv_i \in E(G)}f(uv_i)\\
         &= (3n+m+1) + 2nr + w_f(v_i).
\end{align*}
 Which is independent of $i$. Hence, $g$ is a svtla labeling of $G \circ K_1$ and $\sv(G \circ K_1) \le \sv(G) + 1$.
\end{proof}
Also, in addition to regular graphs, some non-regular graphs follow the inequality obtained in Theorem~\ref{th:sv-corona}. For example: $\sv(K_4 - e) = 3$ and $\sv((K_4 - e) \circ K_1) = 4 \le \sv(K_4 - e) + 1$ (see Figure~\ref{fig:svtla11}). The following question arises naturally.\\
\begin{problem} 
Characterise the graphs $G$ for which $\sv(G \circ K_1) = \sv(G)+1$.
\end{problem}
\begin{problem}
Let $G$ and $H$ be super vertex total local antimagic graphs. Determine the $\sv(G \circ H)$.
\end{problem}
\begin{figure}
    \centering
    \includegraphics[scale = 0.3]{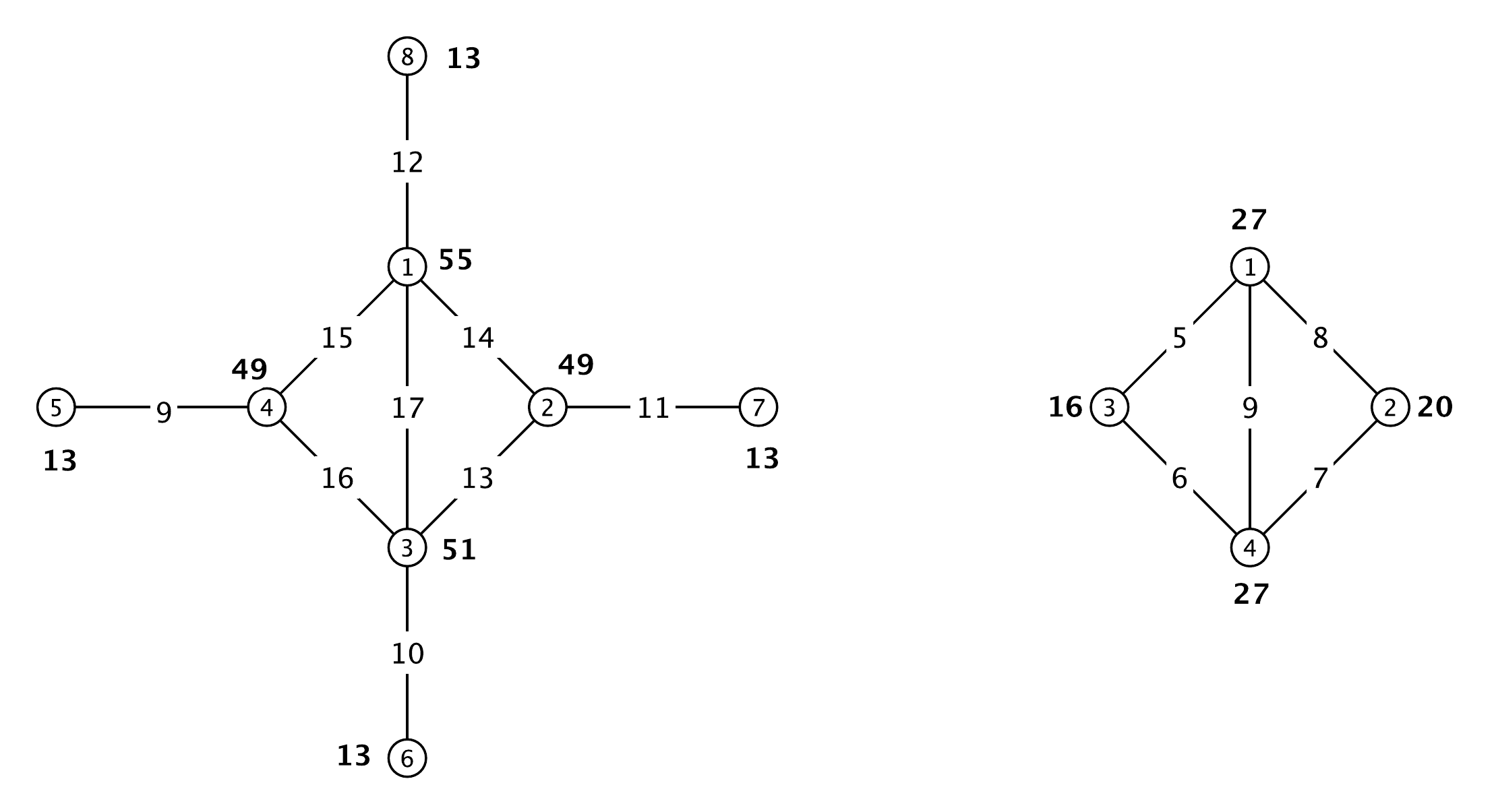}
    \caption{svtla labeling of $(K_4 - e) \circ K_1$ and $K_4 - e$.}
    \label{fig:svtla11}
\end{figure}
\section{Super Edge Total Local Antimagic Labeling}\label{sec:setla}
\begin{proposition} \label{th:se-leaves}
For any graph $G$ with a vertex $v$ having the largest number of pendent vertices $l$, \(\se(G) \ge l+1\).
\end{proposition} 
\begin{proof}
Let $G$ be a graph on $n$ vertices and $v_1, v_2, \dots, v_l$ be pendent vertices at \(v\). Let $f$ be setla labeling. Then the $l$ weights \(w(v_i) = f(v) + f(vv_i)\) are all distinct and $w(v) > w(v_i)$ where $1 \le i \le l$. Hence, $f$ induces $l+1$ colors. This proves the theorem.
\end{proof}
The proof of the following corollary is evident from Proposition \ref{th:sv-leaves}.
\begin{corollary}
For a star $\se(K_{1,n}) = n + 1$.
\end{corollary}
\begin{figure}[ht]
    \centering
    \includegraphics[scale = 0.5]{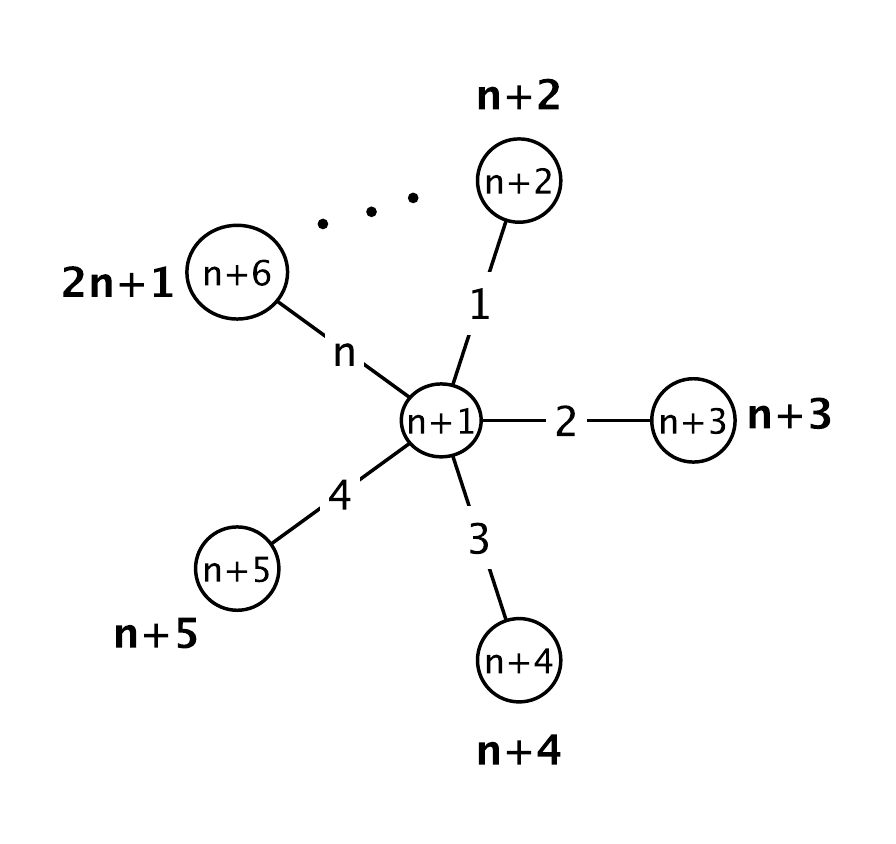}
    \caption{setla labeling of a star.}
\end{figure}
\begin{corollary}
If a tree $T$  with the largest number of pendent vertices equals $l$ at a vertex, then $\se(T) \ge l+1$.
\end{corollary}
\begin{theorem} \label{th: se-path}
For a path $P_n,\; 3 \le \se(P_n) \le 5$.
\end{theorem}
\begin{proof}
It is easy to see that $\se(P_2)=2$ and $\se(P_3) = 3$ (see Figure \ref{fig:sv path}). Let $P_n$ be a path  with vertex set $\{v_1,v_2, \dots, v_n\}$ and $e_i = v_iv_{i+1}$ be edges, where $ 1\le i \le n-1$, where $n \ge 4$. First, we establish the lower bound. Let $f$ setla labeling of $P_n$. Then, $w(v_1) = f(v_2) + f(e_1)$ and $w(v_3) = f(v_2) + f(v_4) + f(e_2) + f(e_3)$. If $w(v_1) = w(v_3)$ then we obtain $f(e_1) = f(v_4) + f(e_2) + f(e_3)$. Which is impossible since $f(e_1) \le n-1$ and $f(v_4) + f(e_2) + f(e_3) \ge n$. Therefore, $w(v_1) \ne w(v_3)$. Also $w(v_1) \ne w(v_2)$ and $w(v_2) \ne w(v_3)$ since $f$ is setla labeling and $v_1v_2, v_2v_3 \in E(P_n)$. This proves that \begin{equation} \label{eq:sepath}
    \se(P_n) \ge 3.
\end{equation}
To prove the upper bound, we consider the following two cases:\\
\\
\textbf{Case 1:} When $n \equiv 0(\bmod~4)$, define super edge total labeling $f$ by\\
\begin{align*}
        f(e_i) &= \begin{cases}
                i-1  \quad & \text{if } i \equiv 0(\bmod ~4)\\
                i  \quad & \text{if } i \equiv 1 \mbox{ or } 2(\bmod~4) \text{ and } i \ne n-2\\
                i+1  \quad & \text{if } i \equiv 3(\bmod~4) \text{ and } i \ne n-1\\
                n-1 \quad & \text{if } i=n-2\\
                n-2 \quad & \text{if } i=n-1
                \end{cases} 
\end{align*}
and $f(v_i) = 2n-i$. Therefore,
$w(v_1) = f(v_2) + f(e_1) = 1+(2n-2) = 2n-1$, $w(v_{n}) = f(v_{n-1}) + f(e_{n-1}) = (n+1) + (n-2) = 2n-1$,\\
$w(v_{n-1}) = f(v_{n-2}) + f(v_n) + f(e_{n-2}) + f(e_{n-1}) = (n+2) + n + (n-1) + (n-2) = 4n-1$,\\
$w(v_{n-2}) = f(v_{n-3}) + f(v_{n-1}) + f(e_{n-3}) + f(e_{n-2}) = (n+3) + (n+1) + (n-2) + (n-1) = 4n$. \\
Now for  each $i, \; 2 \le i \le n-3$,
\begin{align*}
    w(v_i) &= f(v_{i-1}) + f(v_{i+1}) + f(e_{i-1}) + f(e_{i})\\
           &= 2n-(i-1) + 2n-(i+1)+ f(e_{i-1}) + f(e_{i})\\
           &= 4n-2i + f(e_{i-1}) + f(e_{i})\\
           &= 4n-2i + \begin{cases}
            i + (i- 1) &\text{if } i \equiv 0(\bmod~4)\\
           (i-2)+i & \text{if } i \equiv 1(\bmod~4)\\
           (i-1)+i & \text{if } i \equiv 2(\bmod~4)\\
           (i-1)+(i+1) & \text{if } i \equiv 3(\bmod~4), \text{ and } i \ne n
           \end{cases}\\
           &= \begin{cases}
           4n-1 & \text{if $i$ is even}\\
           4n-2 & \text{if } i \equiv 1(\bmod~4)\\
           4n & \text{if } i \equiv 3(\bmod~4).
           \end{cases}
\end{align*}
Observe that $f$ is setla labeling and induces $4$ colors.\\
\\
\textbf{Case 2:} When $n \not \equiv 0(\bmod\ 4)$, define super edge total labeling $f$ by
\begin{align*}
        f(e_i) &= \begin{cases}
                i-1  \quad & \text{ if } i \equiv 0(\bmod~4) \\
                i  \quad & \text{ if } i \equiv 1 \mbox{ or } 2(\bmod~4) \\
                i+1  \quad & \text{ if } i \equiv 3(\bmod~4)
                \end{cases} 
\end{align*}
and $f(v_i) = 2n-i$.
Therefore, $w(v_1) = f(v_2) + f(e_1) = (2n-2) + 1 = 2n-1$,
\begin{align*}
    w(v_n)  &= f(v_{n-1}) + f(e_{n-1})\\
            &= n+1+
            \begin{cases}
            n-2 & \text{if } n \equiv 1(\bmod~4)\\
            n-1 & \text{if } n \equiv 2\mbox{ or } 3(\bmod~4)
            \end{cases}\\
            &=
            \begin{cases}
            2n-1 & \text{if } n \equiv 1(\bmod~4)\\
            2n & \text{if } n\equiv 2 \mbox{ or } 3(\bmod~4)
            \end{cases}
\end{align*}
and for $2 \le i \le n-1$,
\begin{align*}
    w(v_i) &= f(v_{i-1}) + f(v_{i+1}) + f(e_{i-1}) + f(e_{i})\\
           &= [2n-(i-1)] + [2n-(i+1)] + f(e_{i-1}) + f(e_{i})\\
           &= (4n-2i) + f(e_{i-1}) + f(e_{i})\\
           &=(4n-2i) +
           \begin{cases}
            i + (i-1)  & \text{if $i \equiv 0(\bmod~4)$}\\
            (i-2)+ i & \text{if $i \equiv 1(\bmod~4)$}\\
            (i-1)+i & \text{if $i \equiv 2(\bmod~4)$}\\
            (i-1)+(i+1) & \text{if $i \equiv 3(\bmod~4)$}
           \end{cases}\\
           &= \begin{cases}
           4n-1 & \text{if $i$ is even}\\
           4n-2 & \text{if } i \equiv 1(\bmod~4)\\
           4n & \text{if } i \equiv 3(\bmod~4)
           \end{cases}
\end{align*}
\textbf{Subcase (i):} When $n \equiv 1(\bmod~4)$, $w(v_1) = w(v_n) = 2n-1$ and for each $i,~2 \le i \le n-1$, $ w(v_i)$ is $4n-1$ or $4n-2$ or $4n$. Therefore, $\se(P_n) \le 4$.\\
\\
\textbf{Subcase (ii):} When $n \equiv 2 \mbox{ or } 3(\bmod\ 4)$, $w(v_1) = 2n-1$, $w(v_n) = 2n$ and for each $i,~ 2 \le i \le n-1$, $ w(v_i)$ is $4n-1$ or $4n-2$ or $4n$. Therefore, $\se(P_n) \le 5$.\\
Hence, $3 \le \se(P_n) \le 5$.
This completes the proof.
\end{proof}
\begin{figure}[ht]
    \centering
    \includegraphics[scale = 0.6]{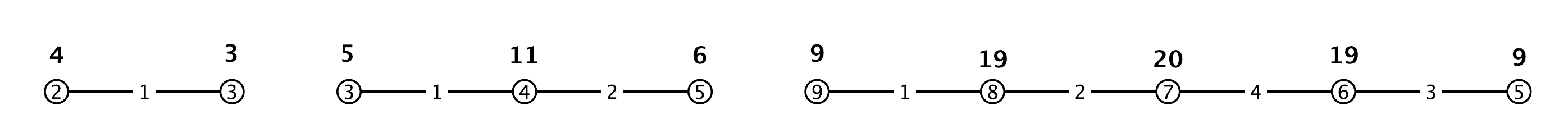}
    \caption{setla labeling of $P_2, P_3$ and $P_5$.}
    \label{fig:sv path}
\end{figure}
\begin{proposition} \label{th:se-c4}
 $\se(C_4) = 3$.   
\end{proposition}
\begin{proof}
 Consider a setla cycle $C_4$ on vertices $v_1, v_2, v_3, v_4$ and with setla labeling $f$. Since, $f$ is setla and $v_1v_2, v_2v_3 \in E(C_4)$, we have $w(v_1) \ne w(v_2)$ and $w(v_2) \ne w(v_3)$. We will show that $\se(C_4) = 3$. On the contrary, suppose $\se(C_4) = 2$. We must have, $w(v_1) = w(v_3)$, which implies $f(e_1) + f(e_4) = f(e_2) + f(e_3)$ and $w(v_2) = w(v_4)$, which implies implies $f(e_1) + f(e_2) = f(e_3) + f(e_4)$. We obtain a contradiction $f(e_1) = f(e_3)$ from the last two equalities. Hence, $\se(C_4) \ge 3$. Also from Figure~\ref{fig:se-cycles}, $\se(C_4) \le 3$. Therefore, $\se(C_4) = 3$.   
\end{proof}
\begin{theorem} \label{th:se-cycle}
For the cycle $C_n, \; n \ge 3, 3 \le \se(C_n) \le 5$.
\end{theorem}
\begin{proof}
It is easy to observe that $\se(C_3) = 3$. By Proposition \ref{th:se-c4}, $\se(C_4) = 3$. Now consider the cycle $C_n$  with vertex set $\{v_1,v_2, \dots, v_n\}$ and $e_i = v_iv_{i+1}$ be edges for $n \ge 5$, where subscripts are taken modulo $n$. We consider the following two cases:\\
\\
\textbf{Case 1:} When $n \equiv 3(\bmod~4)$, define super edge total labeling $f$ by
\begin{align*}
    f(e_i) &= \begin{cases}
                i-1  \quad & \text{ if } i \equiv 0(\bmod~4) \\
                i  \quad & \text{ if } i \equiv 1 \text{ or }2(\bmod~4)\\
                i+1  \quad & \text{if } i \equiv 3(\bmod~4), i \ne n \\
                n , \quad & \text{if } i=n
                \end{cases}\\
    f(v_i) &=\begin{cases}
            2n-i+1 \quad & \text { if } i \ne 1, 2\\
            2n-1 \quad & \text{ if } i = 1\\
            2n \quad & \text{ if } i = 2.
            \end{cases}
\end{align*}
Therefore,
\begin{align*}
w(v_1)  &= f(v_2) + f(v_n) + f(e_1) + f(e_n)\\
        &= 2n + (n+1) + 1 + n\\
        &= 4n+2,\\
w(v_2)  &= f(v_1) + f(v_3) + f(e_1) + f(e_2)\\
        &= (2n-1) + (2n-2) + 1 + 2\\
        &= 4n,\\
w(v_3)  &= f(v_2) + f(v_4) + f(e_2) + f(e_3)\\
        &= 2n + (2n-3) + 2 + 4\\
        &= 4n+3,\\
w(v_n)  &= f(v_1) + f(v_{n-1}) + f(e_{n-1}) + f(e_n)\\
        &= (2n-1) + (n+2) + (n-1) + n\\
        &= 5n.
\end{align*}
Now for each $i, \; 4 \le i \le n-1$,
\begin{align*}
    w(v_i) &= f(v_{i-1}) + f(v_{i+1}) + f(e_{i-1}) + f(e_{i})\\
           &= [2n-(i-1)+1] + [2n-(i+1)+1] + f(e_{i-1}) + f(e_{i})\\
           &= 4n+2-2i + f(e_{i-1}) + f(e_{i})\\
           &= 4n+2-2i + \begin{cases}
            i + (i-1) &\text{if $i \equiv 0(\bmod~4)$}\\
           (i-2)+i & \text{if $i \equiv 1(\bmod~4)$}\\
           (i-1)+i & \text{if $i \equiv 2(\bmod~4)$}\\
           (i-1)+(i+1) & \text{if $i \equiv 3(\bmod~4)$}
           \end{cases}\\
           &= \begin{cases}
           4n+1 & \text{if $i$ is even}\\
           4n & \text{if $i \equiv 1(\bmod~4)$}\\
           4n+2 & \text{if $i \equiv 3(\bmod~4)$}.
           \end{cases}
\end{align*}
This proves that $f$ is setla labeling and induces $5$ colors. Hence, $\se(C_n) \le 5$.\\
\\
\textbf{Case 2:} When $n \not \equiv 3(\bmod~4)$, define super edge total labeling $f$ by
\begin{align*}
    f(e_i) &= \begin{cases}
                i-1  \quad & \text{ if } i \equiv 0(\bmod~4) \\
                i  \quad & \text{ if } i \equiv 1 \text{ or }2(\bmod~4)\\
                i+1  \quad & \text{if } i \equiv 3(\bmod~4)
                \end{cases}
\end{align*}
and
\begin{align*}
    f(v_i) =\begin{cases}
            2n-i+1 \quad & \text { if } i \ne 1, 2\\
            2n-1 \quad & \text{ if } i = 1\\
            2n \quad & \text{ if } i = 2.
            \end{cases}
\end{align*}
Therefore, $w(v_2) = f(v_1)+f(v_3)+f(e_1)+f(e_2) = (2n-1)+(2n-2)+1+2 = 4n$, $w(v_3) = f(v_2)+f(v_4)+f(e_2)+f(e_3) = 2n+ (2n-3) + 2  + 4 = 4n+3$,
\begin{align*}
    w(v_1) &= f(v_2) + f(v_n) + f(e_1) + f(e_n)\\
           &= 2n + (n+1) + 1 + 
           \begin{cases}
           n & \text{if $n \equiv 1 \mbox{ or }2(\bmod~4)$}\;  \mbox{and} \; n\ne 1\\
           n-1 & \text{if $n \equiv 0(\bmod~4)$}
           \end{cases}\\
           &= \begin{cases}
           4n+2 &\text{if $n \equiv 1 \mbox{ or } 2(\bmod~4)$}\\
           4n+1 &\text{if $n \equiv 0(\bmod~4)$},
           \end{cases}\\
    w(v_n) &= f(v_1) + f(v_{n-1}) + f(e_{n-1}) + f(e_n)\\
           &= (2n-1) + (n+2) + f(e_{n-1}) + f(e_n)\\
           &= (3n+1) + 
           \begin{cases}
           2n-1 & \text{if $n \equiv 0(\bmod~4)$}\\
           2n-2 & \text{if $n \equiv 1(\bmod~4)$}\\
           2n-1 & \text{if $n \equiv 2(\bmod~4)$}\\
           \end{cases}\\
           &= \begin{cases}
           5n & \text{if $n \equiv 0 \mbox{ or } 2(\bmod~4)$}\\
           5n-1 & \text{if $n \equiv 1(\bmod~4)$}.
           \end{cases}
\end{align*}
Now for each $i,\;2 \le i \le n-1$,
\begin{align*}
    w(v_i)  &= (4n-2i+2)+
            \begin{cases}
            2i-1 & \text{if }  i \equiv 0(\bmod~4)\\
            2i-2 & \text{if }  i \equiv 1(\bmod~4)\\
            2i-1 & \text{if }  i \equiv 2(\bmod~4)\\
            2i & \text{if } i \equiv 3(\bmod~4)
            \end{cases}\\
            &= 
            \begin{cases}
            4n & \text{if } i \equiv 1(\bmod~4)\\
            4n+1 & \text{if } i \equiv 0\mbox{ or }2(\bmod~4)\\
            4n+2 & \text{if } i \equiv 3(\bmod~4).
            \end{cases}
\end{align*}
Observe that $f$ is setla labeling and induces $4$ colors. Hence, $\se(C_n) \le 4$.\\ This completes the proof. 
\end{proof}
\begin{figure}[ht]
    \centering
    \includegraphics[scale = 0.5]{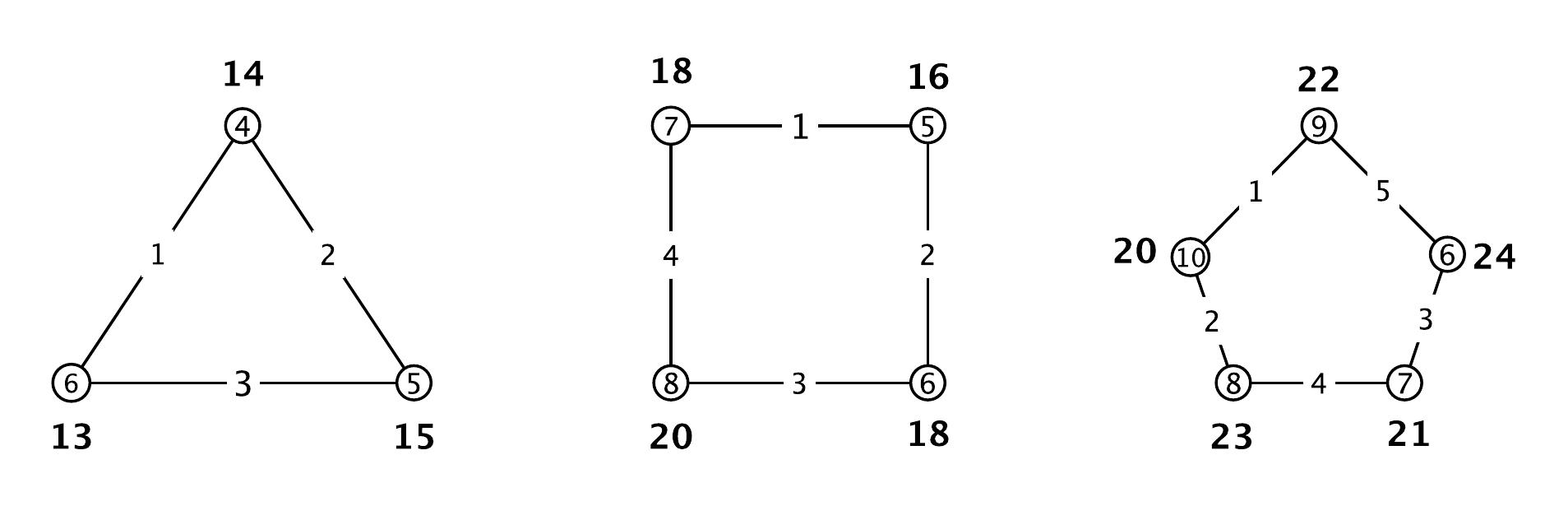}
    \caption{setla labeling of $C_3, C_4$ and $C_5$.}
    \label{fig:se-cycles}
\end{figure}

Now we calculate the setla labeling for a complete bipartite graph $K_{m,n}$ for a few cases. When $m = n =2$, then $K_{2,2} \cong C_4$ and by Proposition \ref{th:se-c4}, $\se(C_4) = 3$ and for $m = 2 $ and $n \ge 3$ we have the following result. \\
\begin{theorem} \label{th:se-bipartite}
$\se(K_{m,n}) = 2$ if and only if $m,n>1, mn >4$ and $m \equiv n (\bmod~2)$.
\end{theorem}
\begin{proof}
Suppose $m,n>1, mn >4$ and $m \equiv n (\bmod~2)$. Then by Theorem \ref{mr}, there exists a magic rectangle $MR(m,n) = [r_{i,j}]_{m\times n}$ with row sum $\rho = \frac{n(mn+1)}{2}$, column sum $ \sigma = \frac{m(mn+1)}{2}$. Let $\{x_1, x_2, \dots, x_m \}$ and $\{y_1, y_2, \dots, y_n \}$ be bipartition of the vertex set of $K_{m,n}$ with $m \le n$. Define the super edge total labeling $f$ by
\begin{align*}
f(a) = \begin{cases}
        mn+i &\mbox{ if } a = x_i\\
        mn+m+i &\mbox{ if } a = y_i\\
        r_{i,j} &\mbox{ if } a = x_iy_j
        \end{cases}
\end{align*}
Then for any $1 \le i \le m$ and for any $1 \le j \le n$,
\begin{align*}
w(x_i) 
&= \sum_{j=1}^{n} f(y_j) + \rho \\
&= n(mn+m) + \frac{n(n+1)}{2} + \frac{n(mn+1)}{2}\\
&= (3m+1)n^2 + (2m+2)n,
\end{align*}
\begin{align*}
w(y_j)
&= \sum_{i=1}^{n} f(x_i) + \sigma\\
&= m(mn) + \frac{m(m+1)}{2} + \frac{m(mn+1)}{2}\\
&= (3n+1)m^2 + 2m.
\end{align*}
It is easy to verify that $w(x_i) > w(y_j)$ for any $i,~1 \le i \le m$ and for any $j,~1 \le j \le n$. Therefore, $f$ is setla labeling and induces $2$ colors. Hence, $\se(K_{m,n}) \le 2$. We know, $\chi(K_{m,n}) = 2$. Therefore, $\se(K_{m,n}) \geq \chi(K_{m,n}) =2$.\\
\\
Conversely, suppose that $\se(K_{m,n}) = 2$ with setla labeling $f$. We must have $w(x_i) \ne w(y_j)$ for any $i,\;1 \le i \le m$ and for any $j,\;1 \le j \le n$ and weights of vertices in independent sets must be same. We can form a magic rectangle $MR(m,n)$ with $(i,j)$th entry as $f(x_iy_j)$ so that row sum is equal to $w(y_1)$ and the column sum is equal to $w(x_1)$. Then again, by Theorem \ref{mr}, $m$ and $n$ satisfies the required conditions. This completes the proof.
\end{proof}
\begin{proposition}
For any $n \ge 1$,
\begin{align*}
\se(K_{2,n}) = 
\begin{cases}
2 \mbox{ if } n \ge 4 \mbox{ and } m \mbox{ is even}\\
3 \mbox{ if } n \mbox{ is odd or } n=2.
\end{cases}
\end{align*}
\end{proposition}   
\begin{proof}
We know that $\se(K_{2,1}) = \se(K_{2,2}) = 3$. Let $n \ge 3$. If $n$ is even by Theorem~\ref{th:se-bipartite}, $\se(K_{2,n}) = 2$. Let $n$ be odd. Again by Theorem~\ref{th:se-bipartite}, $\se(K_{2,n}) \ge 3$. We will show that $\se(K_{2,n}) \le 3$. Let $\{x,y\}$ and $\{u_1, u_2, \dots, u_n\}$ be bipartition of $K_{2,n}$, where $n \ge 3$. Define super edge total labeling of $K_{2,n}$ by
$f(x) = 3n+1, f(y) = 3n+2$, for each $i, 1 \le i \le n$, $f(u_i) = 2n+i$, and 
\begin{align*}
    f(au_i) = \begin{cases}
            i & \text{ if } a = x\\
            2n+1-i & \text{ if } a = y
            \end{cases}
\end{align*}
Now we calculate the weights:
For each  $i,\; 1 \le i \le n$, $w(u_i) = 8n+4$,  $w(x) = 3n^2 + n, w(y) = 4n^2 + n$. It is easy verify that $w(u_i) < w(x) < w(y)$ for all $i, 1 \le i \le n$. Hence, $f$ is setla labeling and induces $3$ colors. Therefore, $\se(K_{2,n}) \le 3$. Hence, $\se(K_{2,n}) = 3$.  
\end{proof}
Then we have the following question:\\
\begin{problem}
Determine the super edge total local antimagic chromatic number of the complete bipartite graph $K_{m,n}$, where $ m \not \equiv n (\bmod\ 2)$.  
\end{problem}
\begin{theorem}
For a bi-star $\se(B_{m,n}) = n+2$.
\end{theorem}
\begin{proof}
Suppose $m \le n$. By the Theorem \ref{th:se-leaves}, $n+1 \le \se(B_{m,n})$. First we show that $\se(B_{m,n}) > n+1$. Let $f$ be any setla labeling of $B_{m,n}$. We know that weights of all pendent vertices $y_i$ are distinct and $w(y) \ne w(y_i)$ for any $i,~1\le i \le n$. If $w(x) = w(y_i)$ for any $i$, then we obtain $f(x) + f(xx_i) = f(x) + f(xy) + \sum_{i=1}^{n}(f(y_i) + f(yy_i))$. This implies, $f(xx_i) = f(xy) + \sum_{i=1}^{n}(f(y_i) + f(yy_i))$ which is not possible. Therefore, $\se(B_{m,n}) \ge n+2$. This proves the lower bound. For an upper bound, we define super edge total labeling $f$ of $B_{m,n}$ as:
\begin{align*}
f(x) &= 2m + 2n + 3\\
f(y) &= m + 2n + 3\\
f(xy) &= m + n + 1\\
f(x_i) &= m + n +1 + i  &\text{ if } 1 \le i \le m\\
f(xx_i) &= i &\text{ if } 1 \le i \le m\\
f(yy_i) &= m + i &\text{ if } 1 \le i \le n
\end{align*}
and we label the vertices $y_i$ by $\{2m+n+2, 2m+n+3, \dots, 2m+2n+2\} - \{m+2n+3\}$ in any manner. The sum of these $y_i$ labels is $\displaystyle \frac{3n^2+4mn+2m+3n-2}{2}$. Now we calculate the weights:
\begin{align*}
	w(x) &= f(y) + f(xy) +  \sum_{i=1}^{m} (f(x_i) + f(xx_i))\\
	        &=  (m+2n+3) +  (m+n+1) + \sum_{i=1}^{m} (m+n+1+2i)\\
	        &= 2m^2+mn+4m+3n+4.\\
   w(y) &= f(x) + f(xy) +  \sum_{i=1}^{n} f(y_i) + \sum_{i=1}^{n}f(yy_i)\\
           &= (2m+2n+3) + (m+n+1) + \frac{3n^2+4mn+2m+3n-2}{2} + \sum_{i = 1}^{n} (m+i)\\
           &= 2n^2+3mn+4m+5n+3.
\end{align*}
The weights of pendent vertices: $w(x_i) = f(x) + f(xx_i) = 2m+2n+3+i$ and $w(y_i) = f(y) + f(yy_i) = 2m+2n+3+i$. Thus $f$ is a required setla which induces $n+2$ colors. Hence, $\se(B_{m,n}) = n+2$.
\end{proof}
\begin{figure}[ht]
    \centering
    \includegraphics[scale = 0.5]{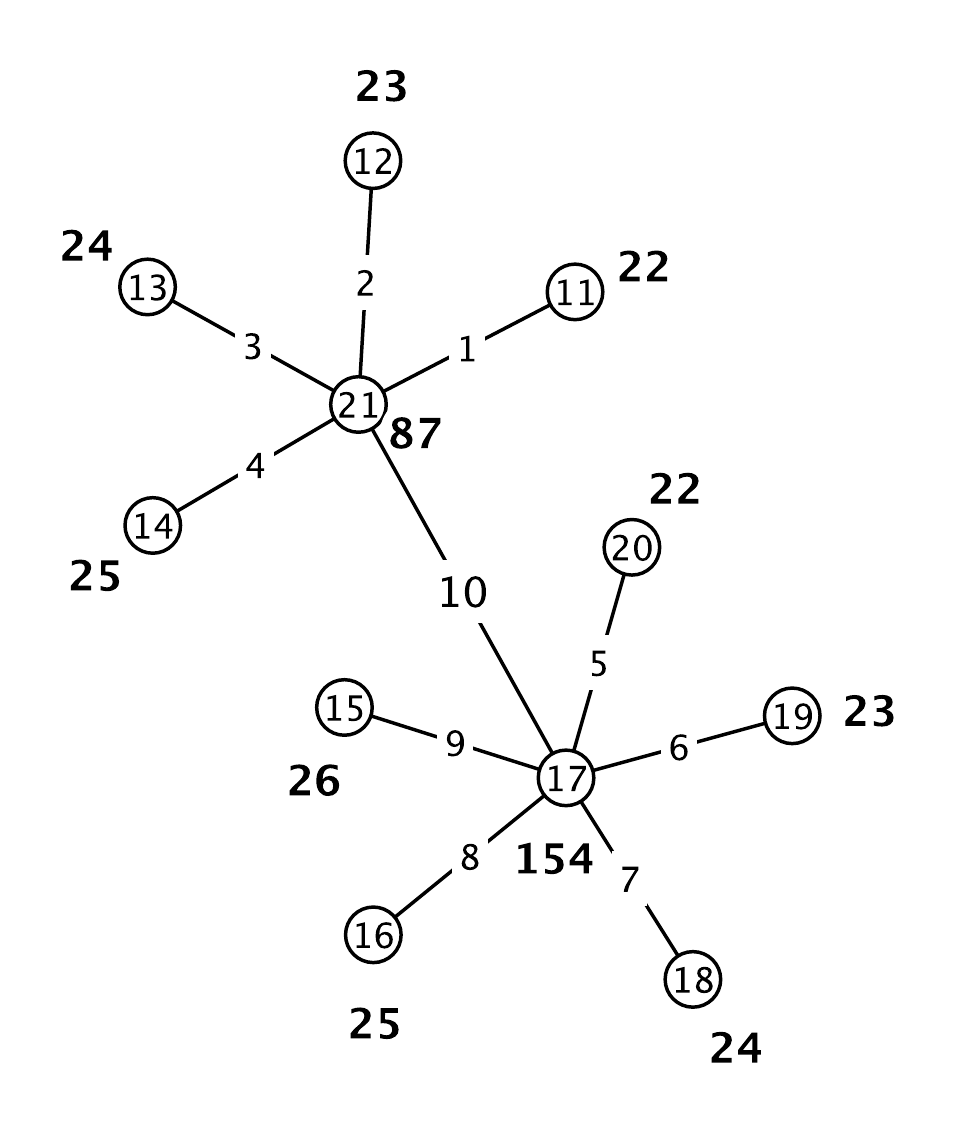}
    \caption{setla labeling of bi-star.}
\end{figure}
\begin{theorem}
If $G$ is $r$-regular graph then $\se(G \circ K_1) \le \se(G) + 1$.
\end{theorem}
\begin{proof}
Let $G$ be an $r$-regular graph with $V(G) = \{v_1, v_2, \dots, v_n\}$ and $m$ edges. Let $f$ be setla labeling of $G$. Without loss of generality, we assume that $f(v_i) = m + i$, where $ 1 \le i \le n$. Let $\{x_1, x_2, \dots, x_n\}$ be newly added vertices to obtain $H = G \circ K_1$ such that there is an edge $e_{i} = v_ix_i$ for each $i,~ 1 \le i \le n$. We define super edge total labeling $g$ of $H$ as follow:\\
\begin{align*}
&g(v_i) = f(v_i) + n = m + n + i &(1 \le i \le n)\\
&g(x_i) = m + 2n + i &(1 \le i \le n)\\
&g(e_i) = m + n + 1 -i &(1 \le i \le m)\\
&g(e) = f(e) & \forall e \in E(G).
\end{align*}
 Now we calculate the weight of each vertex in $G \circ K_1$ induced by $g$.
For each $i,~ 1\le i \le n$, $w_g(x_i) = g(v_i) + g(e_{i}) = 2m + 2n + 1$ and $g(x_i) + g(e_i) = 2m + 3n + 1$. For any vertex  $v_i \in V(G)$, 
\begin{align*}
w_g(v_i) &= \sum_{x \in NT_H(v_i)}g(x)\\
         &= g(x_i) + g(e_i) + \sum_{x \in NT_G(v_i)}g(x)\\
         &= 2m+3n+1 + \sum_{x \in NT_G(v_i)}g(x)\\
         &= 2m+3n+1 + \sum_{u \in N_G(v_i)}g(u) + \sum_{uv_i \in E(G)}g(uv_i)\\
         &= 2m+3n+1 + \sum_{u \in N_G(v_i)}(n+f(u)) + \sum_{uv_i \in E(G)}  f(uv_i)\\
         &= 2m+3n+1 + nr + \sum_{u \in N_G(v_i)}f(u) + \sum_{uv_i \in E(G)}f(uv_i)\\
         &= 2m+3n+1 + nr + w_f(v_i)
\end{align*}
is independent of $i$. This proves that for any $i \ne j$, $w_g(v_i) \ne w_g(v_j)$ if and only if $w_f(v_i) \ne w_f(v_j)$. Hence, $g$ is a setla labeling of $G \circ K_1$ and $\se(G \circ K_1) \le \se(G) + 1$.
\end{proof}
We pose the following problems:\\
\begin{problem}
    Characterize graphs whose $\se(G) = \chi(G)$.
\end{problem}
\begin{problem}
    Characterize graphs whose $\se(G \circ K_1) = \se(G) +1$.
\end{problem}
\begin{problem}
Let $G$ and $H$ be super edge total local antimagic graphs. Determine the $\se(G \circ H)$.
\end{problem}
\begin{conjecture}
Every graph without isolated vertices admits super vertex total local antimagic labeling.
\end{conjecture}
\begin{conjecture}
Every graph without isolated vertices admits super edge total local antimagic labeling.
\end{conjecture}
\section{Conclusion}
We studied and calculated the super vertex (edge) total local antimagic chromatic numbers for a few families of graphs like path, cycle, bipartite graphs etc.
\bibliographystyle{amsplain}
\bibliography{refs}
\end{document}